\documentclass[10pt,reqno]{amsart}
\usepackage{a4wide}
\usepackage{amssymb}
\DeclareMathOperator{\C}{C}
\DeclareMathOperator{\M}{M} \DeclareMathOperator{\Mor}{Mor}
 \DeclareMathOperator{\Rep}{Rep}
\DeclareMathOperator{\D}{D} \DeclareMathOperator{\Aut}{Aut}
\DeclareMathOperator{\Ad}{Ad}
\newtheorem{twr}{Theorem}[section]
\newtheorem{lem}[twr]{Lemma}

\theoremstyle{definition}
\newtheorem{defin}[twr]{Definition}

\newtheorem{uw}[twr]{Remark}

\newcommand{\id}{{\rm id}}

\newcommand{\HG}{\Hat\Gamma}

\newcommand{\CPG}{\C_0(G)^{\tilde\Psi\otimes\Psi}}

\newcommand{\ha}{\hat\alpha}
\newcommand{\hb}{\hat\beta}
\newcommand{\hd}{\hat\delta}
\newcommand{\hg}{\hat\gamma}
\newcommand{\hx}{\hat{x}}
\newcommand{\hy}{\hat{y}}
\newcommand{\hw}{\hat{w}}
\newcommand{\hrp}{\hat\rho^\Psi}
\newcommand{\CHRC}{\C_0(\mathcal{H})\rtimes\mathbb{C}}
\newcommand{\CGRC}{\C_0(G)\rtimes\mathbb{C}^2}
\newcommand{\CHP}{\C_0(\mathcal{H})^\Psi}
\newcommand{\DHP}{\Delta_{\mathcal{H}}^\Psi}
\newcommand{\CH}{\C_0(\mathcal{H})}
\newcommand{\Dh}{\Delta_{\mathcal{H}}}
\begin{document}
\title[]{Rieffel deformation of group coactions}
\author{P.~Kasprzak}
\address{Department of Mathematical Sciences, University of Copenhagen}
\address{On leave from Department of Mathematical Methods in Physics, Faculty of Physics, Warsaw University}
\thanks{Supported by the Marie Curie Research Training Network Non-Commutative Geometry MRTN-CT-2006-031962}
\thanks{Supported by Geometry and Symmetry of Quantum Spaces,
PIRSES-GA-2008-230836}
 \email{pawel.kasprzak@fuw.edu.pl}
\begin{abstract} Let $G$ be a locally compact group, $\Gamma\subset G$ an abelian subgroup and let $\Psi$ be a continuous $2$-cocycle on the dual group $\HG$. Let $B$ be a $\C^*$-algebra and $\Delta_B\in\Mor(B,B\otimes\C_0(G))$ a continuous right coaction.  Using Rieffel deformation, we can construct a quantum group $(\CPG,\Delta^\Psi)$ and the deformed $\C^*$-algebra $B^\Psi$. The aim of this paper is to present a construction of the continuous coaction $\Delta_B^\Psi$ of the quantum group $(\CPG,\Delta^\Psi)$ on $B^\Psi$. The transition from the coaction $\Delta_B$ to its deformed counterpart $\Delta_B^\Psi$ is nontrivial in the sense that $\Delta_B^\Psi$ contains complete information about $\Delta_B$. In order to illustrate our construction we apply it to the action of the Lorentz group on the Minkowski space obtaining a $\C^*$-algebraic quantum Minkowski space.
\end{abstract}
\maketitle \tableofcontents
\begin{section}{Introduction.} 
The Rieffel deformation is already a well established method of deforming $\C^*$-algebras. In his original approach  M. Rieffel starts from  deformation data $(A,\rho,J)$ which consists of a  $\C^*$-algebra $A$, an $\mathbb{R}^n$ action $\rho$ on $A$ and a skew symmetric matrix $J:\mathbb{R}^n\rightarrow\mathbb{R}^n$. Using these data M. Rieffel was able to deform the original product on the algebra of $\rho$-smooth elements: $A^\infty \subset A$. The deformed $\C^*$-algebra $A^J$ is defined as a $\C^*$-algebraic completion of $A^\infty$ equipped with this deformed product.

Our recent approach to the Rieffel deformation (see \cite{Kasp}) is based on the observation that deformation data $(A,\rho,J)$ enable us to deform an $\mathbb{R}^n$-product structure emerging in the crossed product construction: $(A\rtimes\mathbb{R}^n,\lambda,\hat\rho)$. Namely, using a skew-symmetric matrix $J$ we may deform a dual action obtaining $\hat\rho^J$. The Rieffel deformation $A^J$ is defined as the Landstad algebra of the deformed $\mathbb{R}^n$-product $(A\rtimes\mathbb{R}^n,\lambda,\hat\rho^J)$. This second  approach  generalizes to the cases when the deformation data consist  of $(A,\rho,\Psi)$, where $A$ is a $\C^*$-algebra equipped with an action $\rho$ of an abelian group $\Gamma$ and $\Psi$ is a continuous $2$-cocycle on the dual group $\HG$. The concise account of the deformation procedure (containing the definition of Landstad algebra) is the subject of Section \ref{sec2}.

The Rieffel deformation which was originally developed to deform $\C^*$-algebras can be also used for deforming locally compact groups (see \cite{lan1}, \cite{Rf2}). Namely, let $G$ be a locally compact group, $\Gamma$ a closed abelian subgroup and  $\Psi$ a continuous $2$-cocycle on $\HG$. Using the action of $\Gamma^2$ on $\C_0(G)$ given by the left and the right shifts and defining a particular $2$-cocycle $\tilde\Psi\otimes\Psi$ on $\HG^2$ we may perform the Rieffel deformation of $\C_0(G)$. This leads to  a $\C^*$-algebra $\CPG$, which can be shown to carry the structure of a locally compact quantum group. We shall denote this quantum group by $\mathbb{G}^\Psi=(\CPG,\Delta^\Psi)$.

Let $B$ be a $\C^*$-algebra equipped with a continuous coaction $\Delta_B\in\Mor(B,B\otimes\C_0(G))$. The coaction $\Delta_B$ corresponds to an action $\beta:G\rightarrow\Aut(B)$ of $G$ on $B$. Restricting  this action  to the subgroup $\Gamma\subset G$ we get an action of $\Gamma$ on $B$. This lets us construct  the deformed $\C^*$-algebra $B^\Psi$. The aim of this paper is to show that the coaction $\Delta_B\in\Mor(B,B\otimes\C_0(G))$ can be naturally deformed to a coaction $\Delta_B^\Psi\in\Mor(B^\Psi,B^\Psi\otimes\C_0(G)^{\tilde{\Psi}\otimes\Psi})$. The transition from $\Delta_B$ to $\Delta_B^\Psi$ is nontrivial in the sense that $\Delta_B^\Psi$ contains all the information about $\Delta_B$ so it is possible to reconstruct $\Delta_B$ out of $\Delta_B^\Psi$. 

A particular case of the construction presented in our paper has been discussed in \cite{Varr}. In this paper J. Varilly treats the situation where there is given a pair $\Gamma\subset K\subset G$ of closed subgroups with $\Gamma$ being abelian and $G$ compact. It was then shown that it is possible to perform a covariant deformation of the $\C^*$-algebra $\C_0(K\backslash G)$. The resulting $\C^*$-algebra $\C_0(K\backslash G)^\Psi$ is equipped with a continuous ergodic coaction of the quantum group $\mathbb{G}^\Psi$. In this specific situation the difficulties that one encounters in general do not manifest themselves.  
 
In the final section of our paper we apply our deformation procedure to the action of the Lorentz group (more precisely of $SL(2,\mathbb{C})$) on the Minkowski space $M$. We base the deformation procedure on the subgroup $\Gamma\subset SL(2,\mathbb{C})$ consisting of the diagonal matrices. The quantum Lorentz group that we obtain was already described in \cite{Kasp}. In this paper we focus on the  quantum Minkowski space describing it in terms of the generators of the $\C^*$-algebra $\C_0(M)^\Psi$ which satisfy  $(p,q)$-type commutation relations. In particular, we give a description of the twisted  coaction $\Delta_M^\Psi$ in terms of its action on the generators.  

Throughout the paper we will freely use the language of
$\C^*$-algebras and the theory of locally compact quantum groups.
For the notions of multipliers, affiliated elements, algebras
generated by a family of affiliated elements and the category of
$\C^*$-algebras we refer the reader to \cite{W4} and \cite{W10}. In particular the morphism of $\C^*$-algebras will always be Woronowicz morphisms (see 0.4 \cite{W10}). 
For the theory of locally compact quantum groups we refer to
\cite{kv} and \cite{MNW}. For any subset $X$ of a Banach space $B$, $[X]\subset B$ will denote the closed linear span of $X$.
\end{section}
\begin{section}{Rieffel Deformation via crossed products}\label{sec2}
Throughout the paper we shall use the crossed products approach to the Rieffel deformation. For the detailed treatment of the subject we refer to \cite{Kasp}. In what follows we shall give a concise account of the deformation procedure.
The deformation data $(A,\rho,\Psi)$ consist of a $\C^*$-algebra $A$, an action $\rho$ of an abelian group $\Gamma$ and a continuous $2$-cocycle $\Psi$ on the dual group $\Gamma$. The deformation procedure consists of the following steps:\\
1. Let $B$ be the  crossed product $\C^*$-algebra $B=A\rtimes\Gamma$ and  let $(B,\lambda,\hat\rho)$ be the $\Gamma$-product structure on this  crossed product i.e. $\lambda:\Gamma\rightarrow\M(B)$ is the representation of $\Gamma$ implementing the action $\rho$ and $\hat\rho$ is the dual action on $B$.\\
2. Let $\lambda\in\Mor(\C^*(\Gamma),B)$ be the morphism corresponding to the representation $\lambda\in\Rep(\Gamma,B)$ and let $\Psi_{\hg}\in\M(\C_0(\HG))$ be the family of functions given by $\Psi_{\hg}(\hg')=\Psi(\hg',\hg)$. Applying $\lambda$ to $\Psi_{\hg}$ (note the identification of $\C_0(\HG)$ with $\C^*(\Gamma)$ via the Fourier transform), we get a $\hat\rho$-projective 1-cocycle  $U_{\hg}\in\M(B)$. Using $U_{\hg}$ we define the deformed dual action $\hat\rho^\Psi:\HG\rightarrow\Aut(B)$ by the formula: $\hat\rho_{\hg}(b)=U_{\hg}^*\hat\rho_{\hg}(b)U_{\hg}$,\,\,for any $\hg\in\HG$ and $b\in B$.\\
3. $\C^*$-algebra $A^\Psi$ is defined as the Landstad algebra of the deformed $\Gamma$-product $(B,\lambda,\hat\rho^\Psi)$:
\[A^\Psi=\left\{b\in M(B)\left|\begin{array}{l}1.\,\,\hat\rho^\Psi_{\hg}(b)=b\\
2.\,\,\mbox{The map }\Gamma\ni\gamma\mapsto\lambda_\gamma a\lambda_\gamma^*\in \M(B)\\
\mbox{\,\,\,\,\,\, is norm-continuous}\\
3.\,\,\lambda(x)a,a\lambda(x)\in B \mbox{ for any }x\in \C^*(\Gamma)
\end{array}\right.\right\}\]
\begin{uw}\label{cov}
In the course of this paper we shall use the functorial properties of the Rieffel deformation (see Section 3.2, \cite{Kasp}). Let $(A,\rho,\Psi)$ and $(A',\rho',\Psi)$ be deformation data and let $\pi\in\Mor(A,A')$ be a covariant morphism: $\pi(\rho_\gamma(a))=\rho'_{\gamma}(\pi(a))$ for any $\gamma\in\Gamma$ and $a\in A$. Then there exists a morphism $\pi^\Gamma\in\Mor(A\rtimes\Gamma,A'\rtimes\Gamma)$ that sends $A\subset\M(A\rtimes\Gamma)$ to $A'\subset\M(A'\rtimes\Gamma)$ by means of $\pi$ and which is identity on $\C^*(\Gamma)\subset\M(A\rtimes\Gamma)\cap\M(A'\rtimes\Gamma)$. Restricting $\pi^\Gamma$ to $A^\Psi$ we get a morphism $\pi^\Psi\in\Mor(A^\Psi,A'^\Psi)$.
\end{uw}
\end{section}
\begin{section}{Technicalities}
This section is divided into two parts. The first part deals with some specific situations encountered while performing the Rieffel deformation of $\C^*$-algebras. The results derived here will be needed in a  construction of the deformed coaction $\Delta_B^\Psi$. In the second part we shall  construct the product of $(p,q)$-commuting elements affiliated with a $\C^*$-algebra. This construction will be useful in the description of the quantum Minkowski space which is the subject of the last section.
\begin{subsection}{Technicalities of the Rieffel deformation}
Let $\rho$ be an action of $\Gamma^2$ on a $\C^*$-algebra $A$, such that  $\rho_{\gamma,\gamma}=\id$ for any $\gamma\in\Gamma$. Let $\Psi$ be a continuous $2$-cocycle on $\HG$.
Using $\Psi$ we may define a $2$-cocycle $\tilde\Psi$: \begin{equation}\label{tilpsi}\tilde\Psi(\hg_1,\hg_2)=\overline{\Psi(-\hg_1,-\hg_2)}.\end{equation} The $2$-cocycle on $\HG^2$ that we shall use in the deformation procedure is the tensor product:  $\Psi\otimes\tilde\Psi$. The family of unitaries $U_{\hg_1,\hg_2}\in\M(A\rtimes \Gamma^2)$  is obtained by taking the images of  $\Psi_{\hg_1}\otimes\tilde\Psi_{\hg_2}\in\M(\C_0(\HG)\otimes\C_0(\HG))$ under the morphism $\lambda\in\Mor(\C_0(\HG)\otimes\C_0(\HG),A\rtimes\Gamma^2)$: 
\[U_{\hg_1,\hg_2}=\lambda(\Psi_{\hg_1}\otimes\tilde\Psi_{\hg_2})\in\M(A\rtimes \Gamma^2).\] For the definition of $\Psi_{\hg_1}$ and $\tilde\Psi_{\hg_2}$ see the previous section.
The deformed dual action $\hat\rho^{\Psi\otimes\tilde\Psi}$ is defined by the formula  \[\hat\rho^{\Psi\otimes\tilde\Psi}_{\hat\gamma_{1},\hat\gamma_{2}}(b)=U_{\hg_1,\hg_2}^*\hat\rho_{\hat\gamma_{1},\hat\gamma_{2}}(b)U_{\hg_1,\hg_2}\] for any $b\in A\rtimes\Gamma^2$.
\begin{lem}\label{iso}
Let $(A\rtimes\Gamma^2,\lambda,\hat\rho^{\Psi\otimes\tilde\Psi})$ be the $\Gamma^2$-product introduced above. Then the Landstad algebra $A^{\Psi\otimes\tilde\Psi}$ is isomorphic with $A$. 
\end{lem}
\begin{proof}
Let us introduce an auxiliary function $\Psi^\star:\Hat\Gamma^2\rightarrow \mathbb{T}$  given by 
$\Psi^\star(\hg_1,\hg_2)=\overline{\Psi(\hg_1,-\hg_1-\hg_2)}$. 
Using the $2$-cocycle property of $\Psi$ one can prove that:
\begin{equation}\label{sumup}\Psi^\star(x+u,y+v)=\Psi^\star(x,y)\Psi_{u}(x)\tilde\Psi_{v}(y)\Psi(-x-y,-v)\overline{\Psi(u,-x-y)}\end{equation} for any $x,y,u,v\in\Hat{\Gamma}$.
Using the morphism $\lambda\in\Mor(\C^*(\Gamma^2),A\rtimes\Gamma^2)$ we define a unitary element   $\Upsilon=\lambda(\Psi^\star)$.
The above formula may be expressed in terms of $\Upsilon$ and $U_{\hg_1,\hg_2}$ \begin{equation}\label{Upact}\hat\rho_{\hg_1,\hg_{2}}(\Upsilon)=U_{\hg_1,\hg_2}\,\,\Upsilon\, Z_{\hat\gamma_1,\hat\gamma_2}\end{equation}
where $Z_{\hat\gamma_1,\hat\gamma_2}\in\M(A\rtimes\Gamma^2)$ corresponds to the last two factors in the product on the right hand side of \eqref{sumup}. Let us note that the equality $\rho_{\gamma,\gamma}(a)=a$  implies that $\lambda_{\gamma,\gamma}$ is in the center of the $\M(A\rtimes\Gamma^2)$. Using this fact one can conclude that $\lambda(\Delta(f))$ is in the center of $\M(A\rtimes\Gamma^2)$ for any $f\in\M(\C_0(\HG))$. In particular, $Z_{\hat\gamma_1,\hat\gamma_2}\in\M(A\rtimes\Gamma^2)$ is central for any $\hat\gamma_1,\hat\gamma_2\in\Hat\Gamma$. 

Let us move on to the main part of the proof. For any $a\in A$ we define 
$\pi(a)=\Upsilon a\Upsilon^*\in\M(A\rtimes\Gamma^2)$. Our aim is to show that $\pi(a)\in A^{\Psi\otimes\tilde\Psi}$. In order to do that we must check the three Landstad conditions for an element $\pi(a)\in\M(A\rtimes\Gamma^2)$.\\
1. The invariance of $\pi(a)$ under the twisted dual action $\hat\rho^{\Psi\otimes\tilde\Psi}$ can be checked as follows:
\begin{align*}
\hat\rho^{\Psi\otimes\tilde\Psi}_{\hat\gamma_{1},\hat\gamma_{2}}(\Upsilon a\Upsilon^*)&=
U_{\hg_1,\hg_2}^*\,\hat\rho_{\hg_{1},\hg_{2}}(\Upsilon a\Upsilon^*)\,U_{\hg_1,\hg_2}\\&=\Upsilon Z_{\hat\gamma_1,\hat\gamma_2}aZ_{\hat\gamma_1,\hat\gamma_2}^*\Upsilon^*=\Upsilon a\Upsilon^*=\pi(a),
\end{align*}
where in the last line we used Eq. \eqref{Upact} and the fact that $Z$ is a central element.\\
2. In order to check the second Landstad condition for $\pi(a)$ we note that $\lambda_{\gamma_1,\gamma_2}\pi(a)\lambda_{\gamma_1,\gamma_2}^*=\pi(\lambda_{\gamma_1,\gamma_2}a\lambda_{\gamma_1,\gamma_2}^*)$. This equality, together with the second Landstad condition for $a\in A$, shows that the map $\Gamma^2\ni(\gamma_1,\gamma_2)\mapsto \lambda_{\gamma_1,\gamma_2}\pi(a)\lambda_{\gamma_1,\gamma_2}^*\in\M(A\rtimes\Gamma^2)$ is norm continuous.\\
3. Using the third Landstad condition for $a\in A$ we get \[x\pi(a)y=(x\Upsilon) a(\Upsilon^* y)\in A\rtimes\Gamma^2\] for any $x,y\in \C^*(\Gamma^2)$. \\
The above reasoning shows that $\pi(A)\subset A^{\Psi\otimes\tilde\Psi}$. In order to prove the opposite inclusion we have to switch   $A$ and $A^{\Psi\otimes\tilde\Psi}$, take $\bar\Psi$ as a $2$-cocycle and use Lemma 3.5 of \cite{Kasp}.
\end{proof}
Now let us pass to the second specific situation that may be encountered while performing Rieffel deformation. Assume that we are given two commuting actions $\alpha$ and $\beta$ of locally compact abelian groups  $\Gamma_1$ and $\Gamma_2$ on a $\C^*$-algebra $A$:
\[\alpha_{\gamma_1}\circ\beta_{\gamma_2}(a)=\beta_{\gamma_2}\circ\alpha_{\gamma_1}(a)\] for any $a\in A$, $\gamma_1\in\Gamma_1$ and $\gamma_2\in\Gamma_2$. Fixing $\gamma_2$ in the above formula  we see that the automorphism $\beta_{\gamma_2}$ is  $\alpha$-covariant. Now let $\Psi_1$ and $\Psi_2$ be continuous $2$-cocycles on $\HG_1$ and $\HG_2$ respectively. Using Proposition 3.8 of \cite{Kasp} we get the deformed automorphism: $\beta^{\Psi_1}_{\gamma_2}\in\Aut(A^{\Psi_1})$. By the functorial properties of the Rieffel deformation (see Section 3.2 of \cite{Kasp}) it follows that  $\Gamma\ni\gamma\mapsto\beta^{\Psi_1}_{\gamma}\in\Aut(A^\Psi_1)$ is an action of $\Gamma_2$ on $A^{\Psi_1}$.  The continuity of that action demands an additional reasoning. To check that  for any $a\in A^{\Psi_1}$ the map 
\[\Gamma_2\ni\gamma\mapsto\beta^{\Psi_1}_{\gamma}(a)\in A^{\Psi_1}\] is norm continuous, we have to invoke the averaging map $\mathfrak{E}^{\Psi_1}:\D(\mathfrak{E}^{\Psi_1})\rightarrow \M(A\rtimes \Gamma_1)$ (see Remark 2.5, \cite{Kasp}). Let $a\in A$ and $f_1,f_2\in\C_0(\Hat\Gamma_1)$ be functions of  compact support. The elements of the form $\mathfrak{E}^{\Psi_1}(f_1af_2)$ constitute a dense subset of $A^{\Psi_1}$ and the map \begin{equation}\label{condef} A\ni a\mapsto \mathfrak{E}^{\Psi_1}(f_1af_2)\in A^{\Psi_1}\end{equation} is norm continuous. Using the definition of $\beta^{\Psi_1}$ (see Remark \ref{cov}) we get : \[\beta^{\Psi_1}_{\gamma_2}\left(\mathfrak{E}^{\Psi_1}(f_1af_2)\right)=\mathfrak{E}^{\Psi_1}(f_1\beta_{\gamma_2}(a)f_2).\] The above formula, the continuity of map \eqref{condef} and the continuity of the $\Gamma_2$-action $\beta$, together imply  that $\beta^{\Psi_1}$ is a continuous $\Gamma_2$-action.

A similar reasoning, with the roles of $\Gamma_1$ and $\Gamma_2$ reversed, leads to a continuous $\Gamma_1$-action $\alpha^{\Psi_2}$ on $A^{\Psi_2}$. There arises the natural question  concerning the relation between $A^{\Psi_1}$ and $A^{\Psi_2}$. Before answering it let us note that the two commuting actions $\alpha$ and $\beta$, give rise to an action $\alpha\times\beta$ of $\Gamma_1\times\Gamma_2$ on $A$. Tensoring $\Psi_1$ and $\Psi_2$ we get a $2$-cocycle on $\Hat\Gamma_1\times\Hat\Gamma_2$. A Fubini type theorem for the averaging maps $\mathfrak{E}^{\Psi_1}$, $\mathfrak{E}^{\Psi_2}$ and $\mathfrak{E}^{\Psi_1\otimes\Psi_2}$ and a functorial gymnastics enables us to prove the following lemma.
\begin{lem}\label{comact} Let $\alpha,\beta$ and $\alpha\times\beta$ be actions on $A$ of $\Gamma_1$, $\Gamma_2$ and $\Gamma_1\times\Gamma_2$  respectively, as introduced above.  Let $\beta^{\Psi_1}$ be the $\Gamma_2$-action on $A^{\Psi_1}$ and $\alpha^{\Psi_2}$ be the $\Gamma_1$-action  on $A^{\Psi_2}$. Then $(A^{\Psi_1})^{\Psi_2}=(A^{\Psi_2})^{\Psi_1}=A^{\Psi_1\otimes \Psi_1}$. 
\end{lem}
\begin{uw}\label{uwid} Let $\alpha,\beta$ be as in the above lemma. Let $\gamma\in\Gamma_2$ and $\beta_{\gamma}=\id$. Using the functorial properties of the Rieffel deformation we get $\beta^{\Psi_1}_{\gamma}=\id^{\Psi}=\id$.
\end{uw}
As the last instance of this subsection, let us consider a situation in which we have a $\C^*$-algebra $A$ acted on by a group $\Gamma$ and assume that $B$ and $C$ are subalgebras of $\M(A)$ such that $[BC]=A$.  Assume also that $\rho$ can be restricted to continuous actions on $B$ and on $C$. Note that the embedding of $B$ into $\M(A)$ is a morphism in the sense of Woronowicz: \[A\supset [BA]\supset [BCA]=[AA]=A. \]  Similarly, the embedding of  $C$ into $\M(A)$ is a morphism. Using Proposition 3.8 of \cite{Kasp} we get the embeddings of $B^\Psi$ and $C^\Psi$ into $\M(A^\Psi)$. Our aim is to show that $A^\Psi=[B^\Psi C^\Psi]$. 
\begin{lem}\label{prodalg} Let $A$, $B$ and $C$ be $\C^*$-algebras introduced above.  Then $A^\Psi=[B^\Psi C^\Psi]$. 
\end{lem}
\begin{proof}
The proof of this lemma is an application of Lemma 2.6 of paper \cite{Kasp}. The usage of this lemma will be legitimate if  the  equality $[\C^*(\Gamma)B^\Psi C^\Psi\C^*(\Gamma)]=A\rtimes\Gamma$ holds, which we check below
\[[\C^*(\Gamma)B^\Psi C^\Psi\C^*(\Gamma)]=[\C^*(\Gamma)BC\C^*(\Gamma)]=[\C^*(\Gamma)A\C^*(\Gamma)]=A\rtimes\Gamma.\]
\end{proof}
\end{subsection}
\begin{subsection}{Product of $(p,q)$-commuting pairs}\label{pqprod} Let $(p,q)$ be a pair of positive numbers and $H$ be a Hilbert space.
The notion of a $(p,q)$-commuting pair of normal operators $R,S$, acting on $H$ was introduced in \cite{W3}. The fact that the product $R\circ S$ is a densely defined, closable operator acting on $H$ follows from Proposition 0.1, \cite{W3}.  

Let $A$ be a $\C^*$-algebra. The notion  of a $(p,q)$-commuting pair $(R,S)$ of elements affiliated with $A$ was introduced in Definition 5.7, \cite{Kasp}. In what follows we shall analyze $R\circ S$, showing that it is a densely defined operator acting on the Banach space $A$ and that its closure $RS$ is affiliated with $A$. The proof for $p=1$ with an additional spectral condition imposed on $R$ and $S$ was given in Lemma 2.15, \cite{PuszWor}.

The $z$-transform of an element $T$ affiliated with $A$ will be denoted by $z(T)$:
\[z(T)=T(1+T^*T)^{\frac{1}{2}}\in \M(A).\] For any $s\in\mathbb{C}$, the $z$-transform of $sT\in A^\eta$ will be denoted by  $z_s(T)$. For  notational convenience, we shall define a $(p^2,q^2)$-commuting pair of normal elements. 
\begin{defin} \label{mu2} Let $A$ be a $\C^*$-algebra and let $(R,S)$
be a pair of normal elements affiliated with $A$. We say that $(R,S)$
is a $(p^2,q^2)$-commuting pair if
\begin{itemize}
\item[1.]$z(R)z(S^*)=z_{pq}(S^*)z_{q/p}(R)$
\item[2.]$z_{q/p}(R)z(S)=z_{pq}(S)z(R)$.
\end{itemize}
The set of all $(p^2,q^2)$-commuting pairs of normal elements affiliated with a $\C^*$-algebra $A$ is denoted by $D_{p^2,q^2}(A)$.
\end{defin}
\begin{twr}\label{prodpq} Let $A$ be a $\C^*$-algebra and let $(R,S)$ be a $(p^2,q^2)$-commuting pair of elements affiliated with $A$. Then the operator $R\circ S:\mathcal{D}(R\circ S)\rightarrow A$  is closable and its closure $RS$ is affiliated with $A$. 
\end{twr}
\begin{proof} The proof of our theorem is based on Theorem 2.3 of \cite{W10}, which describes a correspondence between elements affiliated with $A$ and a subset of $2\times 2$ matrices of elements of $\M(A)$. Let $Q\in\M_2(\mathbb{C})\otimes\M(A)$:
\[Q=\left(\begin{array}{cc}d&-c^*\\
														b&a^*\end{array}\right).\] The affiliated element $T\,\eta A$ is related with $Q$ by the following correspondence.
The first column of $Q$ provides   information about  $T\in A^\eta$, in the sense that $dA$ is a core of $T$ and $Tdx=bx$ for any $x\in A$. The second column of $Q$ provides information about $T^*\in A^\eta$ in the sense that, $a^*A$ is a core of $T^*$ and $T^*a^*x=c^*x$. The consistency condition $ab=cd$ reflects the defining property of the $*$-operation. In order for $T$ and $T^*$ to be densely defined it is necessary that $dA$ and $a^*A$ are dense in $A$. Finally, to ensure that such a matrix does in fact define an affiliated element $T$, one needs to check that the image of $Q$ as a map acting on the Hilbert module $A_2=A\oplus A$ is dense in $A_2$: $\overline{QA_2}^{||\cdot||}=A_2$. 

Let us move on to the main part of the proof, which was inspired by the proof of Theorem 6.1 of \cite{WorNap}.
Let $(R,S)$ be a  $(p^2,q^2)$ commuting pair of elements affiliated with $A$. The matrix $Q$, which will let us define $RS$, has the form:
\[Q=\begin{pmatrix}
  \left(1-z_{p/q}(R)^*z_{p/q}(R)\right)^{\frac{1}{2}}(1-z(S)^*z(S))^{\frac{1}{2}} & -z(S)^*z(R)^* \\
  z(R)z(S) & (1-z(R)^*z(R))^{\frac{1}{2}}(1-z_{pq}(S)^*z_{pq}(S))^{\frac{1}{2}} 
\end{pmatrix}.\]
The only nontrivial condition to check is that $QA_2$ is dense in $A_2$. In order to do that 
let us consider the matrix $QQ^*\in M_2(\mathbb{C})\otimes\M(A)$. We start with a computation of the left upper corner of $QQ^*$
\begin{align*}(QQ^*)_{11}&=\left(1-z_{p/q}(R)^*z_{p/q}(R)\right)(1-z(S)^*z(S))+z(S)^*z(R)^*z(R)z(S)\\&=\left(1-z_{p/q}(R)^*z_{p/q}(R)\right)(1-z(S)^*z(S))+z_{p/q}(R)^*z_{p/q}(R)z(S)^*z(S)
\end{align*}
where in the last equality we used Definition \ref{mu2} to commute $z(S)^*$ with $z(R)^*z(R)$.
Similarly, we compute the right bottom corner of $QQ^*$:
\[(QQ^*)_{22}=(1-z(R)^*z(R))(1-z_{pq}(S)^*z_{pq}(S))+z(R)^*z(R)z_{pq}(S)^*z_{pq}(S).\]
The matrix elements $(QQ^*)_{12}=(QQ^*)^*_{21}$ turn out to be $0$:
\begin{align*}
(QQ^*)_{12}&=\left(1-z_{p/q}(R)^*z_{p/q}(R)\right)^{\frac{1}{2}}(1-z(S)^*z(S))^{\frac{1}{2}}z(S)^*z(R)^*\\&-z(S)^*z(R)^*(1-z(R)z(R)^*)^{\frac{1}{2}}(1-z_{pq}(S)z_{pq}(S)^*)^{\frac{1}{2}}=0.
\end{align*} To show that the above difference is zero we used the following two identities:
\begin{align*}
(1-z(S)^*z(S))^{\frac{1}{2}}z(R)^*&=z(R)^*(1-z_{pq}(S)z_{pq}(S)^*)^{\frac{1}{2}}\\
z(S)^*(1-z(R)z(R)^*)^{\frac{1}{2}}&=z(S)^*\left(1-z_{p/q}(R)^*z_{p/q}(R)\right)^{\frac{1}{2}},
\end{align*} which immediately follows from Definition \ref{mu2}.

Let us note that \[ (QQ^*)_{11}=\frac{1+|R|^2|S|^2}{(1+|R|^2)(1+|S|^2)}.\] In particular the right ideal generated by $(QQ^*)_{11}$ is dense in $A$: $\overline{(QQ^*)_{11}A}^{||\cdot||}=A$. The same concerns $(QQ^*)_{22}$. From the fact that $QQ^*$ is a diagonal matrix we can see that $\overline{QQ^*A_2}^{||\cdot||}=A_2$, which implies that $QA_2$ is a dense subset of $A_2$. This in turns shows that $Q$ satisfies the assumptions  of Theorem 2.3 of paper \cite{W10}. We conclude that $Q$ defines an affiliated element $T\,\eta A$ such that the set
\begin{equation}\label{coret}\left(1-z_{p/q}(R)^*z_{p/q}(R)\right)^{\frac{1}{2}}(1-z(S)^*z(S))^{\frac{1}{2}}A\end{equation} is a core of $T$ and \[T\left(1-z_{p/q}(R)^*z_{p/q}(R)\right)^{\frac{1}{2}}(1-z(S)^*z(S))^{\frac{1}{2}}a=z(R)z(S)a.\]
Finally, using point 2 of Theorem 2.3 of \cite{W10} one can check that the set 
\[\mathcal{D}(R\circ S)=\{a\in\mathcal{D}(S): Sa\in\mathcal{D}(R)\}\] is a subset of $\mathcal{D}(T)$ and $T|_{\mathcal{D}(R\circ S)}=R\circ S$, hence we may conclude that $R\circ S$ is a closeable operator acting on $A$ and $(R\circ S)^{\mbox{cl}}\subset T$. In order to get the opposite inclusion it is enough to note that the core \eqref{coret} of $T$ is contained in $\mathcal{D}(R\circ S)$.
\end{proof}
\begin{twr}\label{twrs} Let $(R,S)$ be a $(p^2,q^2)$-commuting pair of normal elements affiliated with a $\C^*$-algebra $A$. Then 
\begin{align*} RS&=p^2SR\\
RS^*&=q^2S^*R
\end{align*}
\end{twr}
\begin{proof} In order to prove the first relation let us observe that for any positive number $t>0$ and any affiliated element $X\,\eta A$  we have  \[(1-z_t(X)^*z_t(X))^{\frac{1}{2}}A=(1-z(X)^*z(X))^{\frac{1}{2}}A.\] This shows that the set \[(1-z(R)^*z(R))^{\frac{1}{2}}(1-z(S)^*z(S))^{\frac{1}{2}}A\] is a joint core for $RS$ and $SR$. Using the $(p^2,q^2)$-commutation relations for the pair $(R,S)$ we get 
\begin{align*}RS(1-z(R)^*z(R))^{\frac{1}{2}}(1-z(S)^*z(S))^{\frac{1}{2}}&=\frac{p}{q}z_{q/p}(R)z(S)\\
SR(1-z(R)^*z(R))^{\frac{1}{2}}(1-z(S)^*z(S))^{\frac{1}{2}}&=\frac{1}{pq}z_{pq}(S)z(R)
\end{align*}
From Definition \ref{mu2} we see that the right hand sides of the above equations are equal up to a multiplicative constant $p^2$. This ends the proof of the equality $RS=p^2SR$. The second equality $RS^*=q^2S^*R$ can be proved in a similar way.
\end{proof}
\end{subsection}\end{section}
\begin{section}{Rieffel deformation of group coactions}\label{rdc}
In this section we shall describe the Rieffel deformation of continuous group coactions. We adopt the following definition.
\begin{defin} Let $(A,\Delta)$ be a locally compact quantum group and let $B$ be a $\C^*$-algebra. A morphism $\Delta_B\in\Mor(B,B\otimes A)$ is said to be a continuous right coaction of $(A,\Delta)$ on $B$ if
\begin{itemize}
\item[(1)] $(\Delta_B\otimes\iota)\Delta_B=(\iota\otimes\Delta)\Delta_B$;
\item[(2)] $[(1\otimes A)\Delta_B(B)]=B\otimes A$.
\end{itemize}
\end{defin}
There is a one to one correspondence between the continuous coaction  $\Delta_B\in\Mor(B,B\otimes\C_0(G))$ of $(\C_0(G),\Delta)$ and the  continuous action of $G$ on $B$. The action $\beta:G\rightarrow\Aut(B)$ that corresponds to $\Delta_B$ is defined as follows. Let $\chi_g\in\Mor(\C_0(G),\mathbb{C})$ be the character associated with a group element $g\in G$. We define $\beta_g\in\Aut(B)$ by the following formula: $\beta_g(b)=(\iota\otimes\chi_g)\Delta_B(b)$. It is easy to check that this indeed defines a continuous action of $G$ on $B$.
 
In order to perform the Rieffel deformation of $\Delta_B$ let us assume that $G$ contains an abelian subgroup $\Gamma\subset G$ and let $\Psi$ be a continuous $2$-cocycle on the dual group $\HG$. Restricting $\beta:G\rightarrow\Aut(B)$ to the subgroup $\Gamma\subset G$ we get an action of $\Gamma$ on $B$, which shall be denoted by $\alpha:\Gamma\rightarrow\Aut(B)$.  Let $\mu$ and $\nu$ be the actions of $\Gamma$ on $\C_0(G)$ given by the left and the right shifts along $\Gamma$:\[\mu_{\gamma}(f)(g)=f(\gamma^{-1}g),\,\,(\nu_\gamma f)(g)=f(g\gamma) \] for and $g\in G$, $\gamma\in\Gamma$ and $f\in\C_0(G)$. Using the deformation data $(B,\alpha,\Psi)$ we may construct the $\C^*$-algebra $B^\Psi$ and using $(\C_0(G),\mu\times\nu,\tilde\Psi\otimes\Psi)$ (for the notation  $\mu\times\nu$ see Lemma \ref{comact}) we may construct the quantum group $\mathbb{G}^\Psi=(\C_0(G)^{\tilde\Psi\otimes\Psi},\Delta^\Psi)$.  

Let us move on to the construction of the coaction $\Delta_B^\Psi$ of $\mathbb{G}^\Psi$ on $B^\Psi$. In order to do that  we  define an auxiliary $\C^*$-algebra: \[D=[\Delta_B(B)(1\otimes\C_0(\Gamma\backslash G)]\subset\M(B\otimes\C_0(G)),\] where we treat  $\C_0(\Gamma\backslash G)$ as a subalgebra of $\M(\C_0(G))$. The embedding $D\subset\M(B\otimes\C_0(G))$ is non-degenerate: $[D(B\otimes\C_0(G))]=B\otimes\C_0(G)$. For the ease of reference we describe other properties of $D$ in the following lemma.
\begin{lem}\label{yprop}
Let $D$ be the $\C^*$-algebra defined above. Then:
\begin{itemize}
\item[(1)] the $\Gamma^3$-action $\alpha\otimes(\mu\times\nu)$ on $\M(B\otimes\C_0(G))$ restricts to  a continuous action on $D$;
\item[(2)] the image of the coaction $\Delta_B\in\Mor(B,B\otimes\C_0(G))$  is contained in $\M(D)$ and  $\Delta_B\in\Mor(B,D)$; 
\item[(3)] the embedding of $\C_0(\Gamma\backslash G)\subset\M(D)$ given by \[\C_0(\Gamma\backslash G)\ni f\mapsto(1\otimes f)\in\M(D)\] is a Woronowicz morphism.
\end{itemize}
\end{lem} 
\begin{proof} The only  point which is not obvious is the first one. To prove it let us note that $D$ is isomorphic with $B\otimes\C_0(\Gamma\backslash G)$, where the isomorphism $\Phi\in\Mor(B\otimes\C_0(\Gamma\backslash G),D)$ is given by $\Phi(b\otimes f)=\Delta_B(b)(1\otimes f)\in D$ for any $b\in B$ and $f\in\C_0(\Gamma\backslash G)$. Furthermore, it may be checked that \begin{equation}\label{Phisom}(\alpha_{\gamma_1}\otimes(\mu_{\gamma_2}\times\nu_{\gamma_3}))\Phi=\Phi(\alpha_{\gamma_1\gamma_2^{-1}}\otimes\nu_{\gamma_3})\end{equation}  for any $\gamma_1,\gamma_2,\gamma_3\in\Gamma$. This shows that the $\Gamma^3$-action $\alpha\otimes(\mu\times\nu)$ on $\M(B\otimes\C_0(G))$ restricts to a continuous action on $D\subset\M(B\otimes\C_0(G))$.  \end{proof}
Let us move on to the construction of $\Delta_B^\Psi\in\Mor(B^\Psi,B^\Psi\otimes\C_0(G)^{\tilde{\Psi}\otimes\Psi})$. By  Theorem 3.8 the intertwining properties of $\Delta_B$ 
\[\Delta_B(\alpha_\gamma(b))=(\id\otimes\nu_{\gamma})\Delta_B(b)\] enable us  to define   $\Delta_B^\Gamma\in\Mor(B^\Psi,D^\Psi)$ (we wish to keep the symbol $\Delta_B^\Psi$ for a morphism which will be defined later). Consider the $\Gamma^2$-action $\alpha\otimes\mu$ on $D$. It has the following properties:
\begin{itemize}
\item[(1)] $\alpha_\gamma\otimes\mu_\gamma=\id$,
\item [(2)] $\id\otimes\nu$ and $\alpha\otimes\mu$ mutually commute.
\end{itemize}
The first equality follows from \eqref{Phisom}, whereas the second one is obvious. 
Using Lemma \ref{comact} we see that the Rieffel deformation $D^{\Psi\otimes\tilde{\Psi}\otimes\Psi}$ obtained by the $\Gamma^3$-action introduced in  point 1 of Proposition \ref{yprop}, is isomorphic with $(D^{\Psi})^{\tilde{\Psi}\otimes\Psi}$. This in turn, by Lemma \ref{iso} and Remark \ref{uwid} is isomorphic with $D^{\Psi}$. Therefore, composing the morphism  $\Delta_B^\Gamma\in\Mor(B^\Psi,D^\Psi)$ with the isomorphism $D^\Psi\cong D^{\Psi\otimes\tilde{\Psi}\otimes\Psi}$ of Lemma \ref{iso} we may define the morphism:
\begin{equation}\label{uptwist}\Delta_B^\Psi\in\Mor(B^\Psi,D^{\Psi\otimes\tilde{\Psi}\otimes\Psi}):\,\,\,\Delta_B^\Psi(a)=\Upsilon\Delta_B^\Gamma(a)\Upsilon^*.\end{equation}
Finally, the embedding $\iota\in\Mor(D,B\otimes\C_0(G))$ is $\Gamma^3$-covariant (see Lemma \ref{yprop}), which by Remark \ref{cov} of \cite{Kasp} gives a  morphism $\iota^{\Psi\otimes\tilde{\Psi}\otimes\Psi}$ that embeds  $D^{\Psi\otimes\tilde{\Psi}\otimes\Psi}$ into $\M(B^\Psi\otimes\C_0(G)^{\tilde{\Psi}\otimes\Psi})$. Composing $\Delta_B^\Psi$ with this morphism we define $\Delta_B^\Psi\in\Mor(B^\Psi,B^\Psi\otimes\C_0(G)^{\tilde{\Psi}\otimes\Psi})$. 

Note that the construction of $\Delta_B^\Psi$ is obtained by the $\Upsilon$ twist of $\Delta_B^\Gamma$ (see Eq. \eqref{uptwist}). Furthermore, $\Delta_B^\Gamma$ extends naturally to the morphism $\Delta_B^\Gamma\in\Mor(B\rtimes\Gamma,B\otimes(\C_0(G)\rtimes\Gamma))$. Noting that $\Delta_B^\Gamma$ contains the information about $\Delta_B$ we see that the transition $\Delta_B\rightsquigarrow\Delta_B^\Psi$ is nontrivial: using $\Delta_B^\Psi$ and $\Psi$ we may recover $\Delta_B$. 

\begin{twr} The morphism $\Delta_B^\Psi\in\Mor(B^\Psi,B^\Psi\otimes\C_0(G)^{\tilde{\Psi}\otimes\Psi})$ defined above is a right coaction of $\mathbb{G}^\Psi$ on $B^\Psi$:
\begin{equation}\label{coactid}(\Delta_B^\Psi\otimes\id)\Delta_B^\Psi=(\id\otimes\Delta^\Psi)\Delta_B^\Psi.\end{equation}
\end{twr}
\begin{proof}
It follows from the construction above that the morphism $\Delta_B^\Psi\in\Mor(B^\Psi,B^\Psi\otimes\C_0(G)^{\tilde{\Psi}\otimes\Psi})$ is a restriction of the crossed product morphisms $\Ad_{\Upsilon}\circ\Delta_B^\Gamma\in\Mor(B\rtimes\Gamma,B\rtimes\Gamma\otimes\C_0(G)\rtimes\Gamma^2)$. Similarly, the comultiplication $\Delta^\Psi$ is a restriction of the crossed product morphism $\Ad_{\Upsilon}\circ\Delta^\Gamma\in\Mor(\C_0(G)\rtimes\Gamma^2,\C_0(G)\rtimes\Gamma^2\otimes\C_0(G)\rtimes\Gamma^2)$ (see Theorem 3.8 \cite{Kasp}). For the purpose of this proof these crossed product morphisms will also be denoted by $\Delta_B^\Psi$ and $\Delta^\Psi$ respectively. 

We shall prove the  coaction identity \eqref{coactid} on the level of crossed products, which implies the same equality on the level of the deformed algebras. From the fact $B\rtimes\Gamma=[B\C^*(\Gamma)]$ we see that it is enough to check \eqref{coactid} on $B$ and $\C^*(\Gamma)$ separately. Let $\lambda_\gamma\in\C^*(\Gamma)$ be a unitary generator of $\C^*(\Gamma)$. In order to check that \eqref{coactid} holds on $\C^*(\Gamma)$ it is enough to note that
\begin{equation}\label{delcr}(\Delta_B^\Psi\otimes\id)\Delta_B^\Psi(\lambda_\gamma)=(1\otimes 1\otimes\lambda_{e,\gamma})=(\id\otimes\Delta^\Psi)\Delta_B^\Psi(\lambda_\gamma).\end{equation}
Now for any $b\in B\subset\M(B\rtimes\Gamma)$ we have 
\begin{align*}
(\Delta_B^\Psi\otimes\id)\Delta_B^\Psi(b)&=(\Delta_B^\Psi\otimes\id)(\Upsilon\Delta_B(b)\Upsilon^*)\\
&=\Upsilon_{12}\Upsilon_{23}((\Delta_B\otimes\id)\Delta_B(b))\Upsilon_{23}^*\Upsilon_{12}^*.
\end{align*}
On the other hand: 
\begin{align*}
(\id\otimes\Delta^\Psi)\Delta_B^\Psi(b)&=(\id\otimes\Delta^\Psi)(\Upsilon\Delta_B(b)\Upsilon^*)\\
&=\Upsilon_{12}\Upsilon_{23}((\id\otimes\Delta)\Delta_B(b))\Upsilon_{23}^*\Upsilon_{12}^*.
\end{align*}
Using the coaction equation for $\Delta_B$  we get 
$(\Delta_B^\Psi\otimes\id)\Delta_B^\Psi(b)=(\id\otimes\Delta^\Psi)\Delta_B^\Psi(b)$ for any $b\in B$. This together with \eqref{delcr} ends the proof.
\end{proof}
\begin{twr} The coaction $\Delta_B^\Psi$ of $\mathbb{G}^\Psi$ on $B^\Psi$ defined above is continuous.
\end{twr}
\begin{proof}
Let us first note that using Lemma \ref{prodalg} and point (3) of Lemma \ref{yprop} we get
\[[\Delta_B(B)^\Psi(1\otimes\C_0(\Gamma\backslash G)^{\Psi})]=D^{\Psi}.\]
It can be checked that $\Delta_B(B)^\Psi=\Delta_B^\Gamma(B^\Psi)$ where $\Delta_B^\Gamma$ denotes the Rieffel deformation of $\Delta_B$ treated as a morphism from $B$ to $D$ (see the first paragraph on page 8). Applying the isomorphism $D^{\Psi}\rightarrow D^{\Psi\otimes\tilde\Psi\otimes\Psi}$ of Lemma \ref{iso} we get
 \[[\Delta_B^\Psi(B^\Psi)(1\otimes\C_0(\Gamma\backslash G)^{\tilde\Psi\otimes\Psi})]=D^{\Psi\otimes\tilde{\Psi}\otimes\Psi}.\] Moreover, the embedding of $\C_0(\Gamma\backslash G)^{\tilde\Psi\otimes{\Psi}}$ into $\M(\CPG)$ is non-degenerate, hence we see that
\begin{equation}\label{delcon}\begin{array}{rl}[\Delta_B^\Psi(B^\Psi)(1\otimes\CPG)]&\hspace*{-0,25cm}=[\Delta_B^\Psi(B^\Psi)(1\otimes\C_0(\Gamma\backslash G)^{\Psi\otimes\tilde{\Psi}})(1\otimes\CPG)]\\&\hspace*{-0,25cm}=[D^{\Psi\otimes\tilde{\Psi}\otimes\Psi}(1\otimes\CPG)].\end{array}\end{equation}
The continuity of $\Delta_B$: $[\Delta_B(B)(1\otimes\C_0(G))]=[B\otimes\C_0(G)]$ shows that \[[D (1\otimes\C_0(G))]=[D \Delta_B(B)(1\otimes\C_0(G))]=B\otimes\C_0(G).\] In the first equality above we used the non-degeneracy of $\Delta_B\in\Mor(B,D)$: $ [D\Delta_B( B)]=D$. Applying  Lemma \ref{prodalg} to  $D$ and $1\otimes\C_0(G)$ we see that 
\[D^{\Psi\otimes\tilde{\Psi}\otimes\Psi}(1\otimes\C_0(G)^{\tilde\Psi\otimes\Psi})=B^\Psi\otimes\C_0(G)^{\tilde\Psi\otimes \Psi}.\] This together with Eq. \eqref{delcon} ends the proof.
\end{proof}
\end{section} 
\begin{section}{Quantum Minkowski space}
In this section we shall apply the Rieffel deformation to the action of the Lorentz group $G$ (more precisely of $SL(2,\mathbb{C})$) on the Minkowski space $M$, obtaining a $\C^*$-algebraic quantum Minkowski space as a result. For a survey on the quantum Minkowski spaces on a purely algebraic level we refer to \cite{Pod}. There exists a natural extension of our construction to an action of a quantum Poincar\'e group, which gives a $\C^*$-algebraic realization of a family of the quantum Minkowski spaces described in \cite{Pod}. 

As usual, we shall identify $M$ with the set $\mathcal{H}$ of $2\times 2$ hermitian matrices:
\[M\ni(x_0,x_1,x_2,x_3)\mapsto\left(
\begin{array}{cc}x_0+x_3&x_1+ix_2\\
x_1-ix_2&x_0-x_3
\end{array}
\right)\in\mathcal{H}.\] Under this identification the right action of $G$ on $\mathcal{H}$ is given by:
\[\mathcal{H}\times G\ni(h,g)\mapsto g^*hg\in \mathcal{H}.\] 
To perform the Rieffel deformation we use the  subgroup $\Gamma\subset G$ consisting of the 
 diagonal matrices:
\begin{equation}\label{formact}\Gamma=\left\{\left(\begin{array}{cc}e^z&0\\0&e^{-z}\end{array}\right):z\in\mathbb{C}\right\}.\end{equation}
 Our choice of $\Gamma$ is the same as in \cite{Kasp}.
We shall pull back the actions of $\Gamma$ to the actions of $\mathbb{C}$, which is possible due to the morphism $\sigma:\mathbb{C}\rightarrow \Gamma$ given by
 \begin{equation}\label{homrho}\mathbb{C}\ni z\mapsto \sigma(z)=\left(\begin{array}{cc}e^z&0\\0&e^{-z}\end{array}\right)\in\Gamma.\end{equation} To be more precise, if $\alpha:\Gamma\rightarrow \Aut(A)$ is an action of $\Gamma$ on a $\C^*$-algebra $A$, then the formula $\mathbb{C}\ni z\mapsto\alpha_{\sigma(z)}\in\Aut(A)$ defines an action of $\mathbb{C}$ on $A$. It can be shown that all of the constructions of this paper can be performed in the case where we use a continuous group homomorphism  $\sigma:\Gamma\rightarrow G$ instead of a pure embedding $\Gamma\subset G$. 
 
The benefits of pulling back the actions of $\Gamma$ to the actions of $\mathbb{C}$ are related to the self-duality of $\mathbb{C}$ and the simple forms of continuous $2$-cocycles on $\mathbb{C}$. The duality that we shall use in this paper is established by the following bicharacter on $\mathbb{C}$ 
\begin{equation}\label{dualiso}\mathbb{C}^2\ni(z_1,z_2)\mapsto\exp(i\Im(z_1z_2))\in\mathbb{T}.\end{equation}
We shall use the $2$-cocycle $\Psi$ on $\mathbb{C}$ of the form \begin{equation}\label{twococycle}\mathbb{C}^2\ni(z_1,z_2)\mapsto\Psi(z_1,z_2)=\exp(-is\Im(z_1\overline{z}_2))\in\mathbb{T},\end{equation} where $s\in\mathbb{R}$ is the deformation parameter. 
Note that $\Psi$ differs from the $2$-cocycle used in the example presented in the paper \cite{Kasp} by the sign in the exponent. This is related to some sign mistakes that we found in \cite{Kasp} during the preparation of the example for this paper. Some further inconsistencies which the reader may have noticed  are due to the fact that we have corrected the mistakes of \cite{Kasp}. 

Once $\Psi$ and $\Gamma$ have been fixed, we can perform the Rieffel deformation of $\C_0(G)$ and $\C_0(\mathcal{H})$. The analysis of the quantum group $\mathbb{G}^\Psi=(\CPG,\Delta^\Psi)$ was undertaken in \cite{Kasp}. In what follows we shall give a concise description of $\mathbb{G}^\Psi$ in terms of the generators $\ha,\hb,\hg,\hd\,\eta \CPG$. 

Let $\rho$ be the $\mathbb{C}^2$-action on $\C_0(G)$ given by:
\[\rho_{z_1,z_2}(f)(g)=f(\sigma(z_1)^{-1}g\sigma(z_2)),\] where $\sigma$ is the morphism defined by \eqref{homrho}. Using a $2$-cocycle $\tilde\Psi\otimes\Psi$, one can deform the standard  $\mathbb{C}^2$-structure on the crossed product, obtaining $(\C_0(G)\rtimes\mathbb{C}^2,\lambda,\hat\rho^{\tilde\Psi\otimes\Psi})$. It may be checked that our choice of the $2$-cocycle and the  way that we identify $\mathbb{C}$ and $\Hat{\mathbb{C}}$ (see Eq. \eqref{dualiso}) lead to the following formula for the deformed dual action:
\begin{equation}\label{ddmin}\hat\rho^{\tilde\Psi\otimes\Psi}_{z_1,z_2}(b)=\lambda_{-s\bar{z}_1,s\bar{z}_2}\hat\rho_{z_1,z_2}(b)\lambda_{-s\bar{z}_1,s\bar{z}_2}^*.\end{equation} Applying the morphism $\lambda\in\Mor(\C_0(\mathbb{C}^2),\C_0(G)\rtimes\mathbb{C}^2)$ to $\Psi\in\M(\C_0(\mathbb{C}^2))$ we get a unitary element  $U=\lambda(\Psi)\in\M(\C_0(G)\rtimes\mathbb{C}^2)$.  Using $U$ and the coordinate functions $\alpha,\beta,\gamma,\delta$ affiliated with $\C_0(G)$ we define four elements $\ha,\hb,\hg,\hd$ affiliated with $\C_0(G)\rtimes\mathbb{C}^2$:
\begin{equation}\label{defgen}\begin{array}{cc}\ha=U^*\alpha U,&\hb=U\beta U^*\\
										\hg=U\gamma U^*,&\hd=U^*\delta U
										\end{array}.
\end{equation}
The main results of Section 5 of \cite{Kasp} are contained in the following 
\begin{twr}\label{generate}
Let $\ha,\hb,\hg,\hd$ be the elements affiliated $\C_0(G)\rtimes\mathbb{C}^2$ introduced above. Then
\begin{itemize}
\item[1.] $\ha,\hb,\hg,\hd$ are affiliated with $\C_0(G)^{\tilde\Psi\otimes\Psi}$ and they generate it.
\item[2.] They satisfy the following commutation relations:
\[\begin{array}{rcl}
                    \hat{\alpha}\hat{\beta}&\hspace*{-0.25cm}=&\hspace*{-0.25cm}\hat{\beta}\hat{\alpha}\\
                    \hat{\alpha}\hat{\delta}&\hspace*{-0.25cm}=&\hspace*{-0.25cm}\hat{\delta}\hat{\alpha}\\
                    \hat{\alpha}\hat{\gamma}&\hspace*{-0.25cm}=&\hspace*{-0.25cm}\hat{\gamma}\hat{\alpha}\\
                    \hat{\beta}\hat{\gamma}&\hspace*{-0.25cm}=&\hspace*{-0.25cm}\hat{\gamma}\hat{\beta}\\
                    \hat{\beta}\hat{\delta}&\hspace*{-0.25cm}=&\hspace*{-0.25cm}\hat{\delta}\hat{\beta}\\
                    \hat{\gamma}\hat{\delta}&\hspace*{-0.25cm}=&\hspace*{-0.25cm}\hat{\delta}\hat{\gamma}\\
                    \hat{\alpha}\hat{\delta}&\hspace*{-0.25cm}=&\hspace*{-0.25cm} 1+\hat{\beta}\hat{\gamma}

\end{array}
\]\[
\begin{array}{rclrclrccrcl}\hat{\alpha}\hat{\alpha}^*&\hspace*{-0.25cm}=
                    &\hspace*{-0.25cm}\hat{\alpha}^*\hat{\alpha}&&&&&&&&&\\
                    \hat{\alpha}\hat{\beta}^*&\hspace*{-0.25cm}=
                    &\hspace*{-0.25cm}t\hat{\beta}^*\hat{\alpha}&\hat{\beta}\hat{\beta}^*&\hspace*{-0.25cm}=
                    &\hspace*{-0.25cm}\hat{\beta}^*\hat{\beta}&&&&&&\\
                    \hat{\alpha}\hat{\gamma}^*&\hspace*{-0.25cm}=
                    &\hspace*{-0.25cm}t^{-1}\hat{\gamma}^*\hat{\alpha}&\hat{\beta}\hat{\gamma}^*&\hspace*{-0.25cm}=
                    &\hspace*{-0.25cm}\hat{\gamma}^*\hat{\beta}&\hat{\gamma}\hat{\gamma}^*&\hspace*{-0.25cm}=
                    &\hspace*{-0.25cm}\hat{\gamma}^*\hat{\gamma}&&&\\
                    \hat{\alpha}\hat{\delta}^*&\hspace*{-0.25cm}=
                    &\hspace*{-0.25cm}\hat{\delta}^*\hat{\alpha}&\hat{\beta}\hat{\delta}^*&\hspace*{-0.25cm}=
                    &\hspace*{-0.25cm}t^{-1}\hat{\delta}^*\hat{\beta}&\hat{\gamma}\hat{\delta}^*&\hspace*{-0.25cm}=
                    &\hspace*{-0.25cm}t\hat{\delta}^*\hat{\gamma}&\hat{\delta}\hat{\delta}^*&\hspace*{-0.25cm}=
                    &\hspace*{-0.25cm}\hat{\delta}^*\hat{\delta},
\end{array} \] where $t$  in these relations and the deformation parameter $s$ (see \eqref{twococycle}) are related by $t=e^{-8s}$.
\item[3.] The action of  $\Delta^\Psi$  on the generators is given by:
\[\begin{array}{rcl}
\Delta^\Psi(\hat{\alpha})&\hspace*{-0.25cm}=&\hspace*{-0.25cm}\hat{\alpha}\otimes\hat{\alpha}+\hat{\beta}\otimes\hat{\gamma},\\
\Delta^\Psi(\hat{\beta})&\hspace*{-0.25cm}=&\hspace*{-0.25cm}\hat{\alpha}\otimes\hat{\beta}+\hat{\beta}\otimes\hat{\delta},\\
\Delta^\Psi(\hat{\gamma})&\hspace*{-0.25cm}=&\hspace*{-0.25cm}\hat{\gamma}\otimes\hat{\alpha}+\hat{\delta}\otimes\hat{\gamma},\\
\Delta^\Psi(\hat{\delta})&\hspace*{-0.25cm}=&\hspace*{-0.25cm}\hat{\gamma}\otimes\hat{\beta}+\hat{\delta}\otimes\hat{\delta}.
\end{array}\]
\end{itemize}
\end{twr}
Some comments on this theorem are necessary. In point 1 we used the the fact the embedding  $\C_0(G)^{\tilde\Psi\otimes\Psi}\subset\M(\C_0(G)\rtimes\mathbb{C}^2)$ is the Woronowicz morphism and it extends to the embedding of affiliated elements $(\C_0(G)^{\tilde\Psi\otimes\Psi})^\eta\subset(\C_0(G)\rtimes\mathbb{C}^2)^\eta$ (see \cite{W10}). The commutation relations in  point 2 are to be understood in the sense of $(p,q)$-commuting pairs for appropriate $p$ and $q$ (see Definition \ref{mu2}). In particular, by the results of Section \ref{pqprod} the above commutation relations may be understood literally - all elements in these relations exist as elements affiliated with $\CPG$.  The sums used in point 3 denotes the sums of strongly commuting normal elements affiliated with  $\CPG\otimes\CPG$. The summation operation, which in general cannot be defined for a pair of affiliated elements, in this case gives rise   to  normal elements affiliated with $\C_0(G)^{\tilde\Psi\otimes\Psi}\otimes\C_0(G)^{\tilde\Psi\otimes\Psi}$. 
\begin{subsection}{Generators of $\CHP$}
Let us move on to the analysis of the $\C^*$-algebra $\C_0(\mathcal{H})^\Psi$. It is defined as the Landstad algebra of the $\mathbb{C}$-product $(\CHRC,\lambda,\hat\rho^{\Psi})$. The deformed action $\hat\rho^\Psi$ is given by:
\[\hat\rho^\Psi_{z}(b)=\lambda_{s\bar{z}}\hat\rho_{z}(b)\lambda_{s\bar{z}}^*,\] for any $z\in\mathbb{C}$ and $b\in\C_0(\mathcal{H})\rtimes\mathbb{C}$ (compare with \eqref{ddmin}). 

Let $x,y,w\in\C_0(\mathcal{H})^\eta$ be the matrix coefficient functions on the set of hermitian matrices:
\begin{equation}\label{clgen}\mathcal{H}=\left\{\left(\begin{array}{cc}x&w\\\bar{w}&y\end{array}\right):\,x,y\in\mathbb{R},\,w\in\mathbb{C}\right\}.\end{equation} 
It is obvious that $x,y$ and $w$ generate $\C_0(\mathcal{H})$ in the sense of Woronowicz. Our next objective is to  introduce three elements $\hx,\hy,\hw$ generating $\C_0(\mathcal{H})^\Psi$. In order to do that let us introduce a function $\Omega:\,\mathbb{C}\ni z\mapsto\Omega(z)=\exp(-i\frac{s}{2}\Im(z^2))\in\mathbb{T}$ and a unitary element $V\in\M(\CHRC)$, which is the image of $\Omega\in\M(\C_0(\mathbb{C}))$ under  $\lambda\in\Mor(\C_0(\mathbb{C}),\C_0(\mathcal{H})\rtimes\mathbb{C})$: $V=\lambda(\Omega)$. Using $V$ and the coordinate functions $x,y,w\in\C_0(\mathcal{H})^\eta\subset(\CHRC)^\eta$ we define $\hx,\hy,\hw$ as elements affiliated with $\CHRC$:
\begin{equation}\label{defcoef} \hx=e^{-2s}VxV^*,\,\hy=e^{-2s}Vy V^*,\,\hw=e^{2s}V^*wV.
\end{equation} The multiplicative factors $e^{\pm 2s}$ are introduced to get a nice formulas for the  coaction of the quantum Lorentz  group $\mathbb{G}^\Psi$ on $\CHP$.
\begin{twr}\label{chgen}
Let $\hx,\hy$ and $\hw$ be the elements affiliated with $\CHRC$ defined above. Then $\hx,\hy$ and $\hw$ are affiliated with $\C_0(\mathcal{H})^\Psi$ and they generate it. 
\end{twr}
\begin{proof} The proof follows the same line as the proof of the respective theorem concerning elements $\ha,\hb,\hg,\hd$ given in \cite{Kasp}.
Let us  check that $\hx$ is $\hat\rho^\Psi$-invariant:
\begin{equation}\label{defx}\begin{array}{rl}
\hrp_{w}(\hx)&=e^{-2s}\lambda_{s\bar{w}}\hat\rho_{w}(VxV^*)\lambda_{s\bar{w}}^*\\&=e^{-2s}\hat\rho_{w}(V)\,\lambda_{s\bar{w}} x\lambda_{s\bar{w}}^*\,\hat\rho_{w}(V)^*
\end{array}\end{equation}
In order to calculate $\hat\rho_w(V)$ we use the fact that $\lambda\in\Mor(\C_0(\mathbb{C}),\CHRC)$ intertwines the dual action $\hat\rho_{w}$ with the shift action of $\mathbb{C}$ on $\C_0(\mathbb{C})$. It is easy to see that \[\Omega(z+w)=\Omega(z)\Omega(w)\exp(-is\Im(zw)).\] This formula and the way that we identify $\mathbb{C}$ with $\Hat{\mathbb{C}}$ (see Eq. \eqref{dualiso}) enable us to see that 
\[\hat\rho_w(V)=V\Omega(w)\lambda_{-sw}.\] We may now substitute the above equality into \eqref{defx} to obtain
\begin{equation}\label{defx1}\hrp_{w}(\hx)=e^{-2s}V\lambda_{-sw+s\bar{w}}x\lambda_{sw+s\bar{w}}^*V^*.\end{equation} Finally, using the fact that $\lambda_{z}$ implements the action $\rho_{z}$ we get $\lambda_{z} x \lambda_{z}^*=e^{z+\bar{z}}x$. In particular $\lambda_{-sw+s\bar{w}}x\lambda_{sw+s\bar{w}}^*=x$, which substituted into $\eqref{defx1}$ gives \[\hrp_{w}(\hx)=\hx.\] The $\hat\rho^\Psi$-invariance of $\hx\in(\C_0(\mathcal{H})\rtimes\mathbb{C})^\eta$ is a necessary condition  to prove that $\hx$ is affiliated with $\C_0(\mathcal{H})^\Psi$, but it is not sufficient. To this end let us define a morphism $\pi:\C_0(\mathbb{R})\rightarrow\M(\CHRC)$:\,\,
$\pi(f)=f(\hx)$. Obviously, $\pi(f)$ is $\hrp$-invariant for any $f$. Furthermore, the map $\mathbb{C}\ni z\mapsto\lambda_z\pi(f)\lambda_{z}^*\in\M(\CHRC)$ is norm continuous. The latter statement follows from the following computation:
\[\lambda_z\pi(f)\lambda_{z}^*=\lambda_z f(\hx)\lambda_{z}^*=V\lambda_z f(e^{-2s}x)\lambda_{z}^*V^*=Vf(e^{-2s}e^{z+\bar z}x)V^*.\]
Hence we see that $\pi(f)$ satisfies the sufficient conditions to be an element of $\M(\C_0(\mathcal{H})^\Psi)$. 

Let us now show that $\pi\in\Mor(\C_0(\mathbb{R}),\C_0(\mathcal{H})^\Psi)$. In order to do that we have to check the nondegeneracy $[\pi(\C_0(\mathbb{R}))\CHP]=\CHP$. Invoking Lemma 2.6 of \cite{Kasp} it follows from the equality $[\pi(\C_0(\mathbb{R}))\CHP\C^*(\mathbb{C})]=\C_0(\mathcal{H})\rtimes\Gamma$, which we prove as follows:
\begin{align*}[\pi(\C_0(\mathbb{R}))\CHP\C^*(\mathbb{C})]&=[\pi(\C_0(\mathbb{R}))\C_0(\mathcal{H})\rtimes\Gamma]\\
&=[\pi(\C_0(\mathbb{R}))\C^*(\Gamma)\C_0(\mathcal{H})]\\
&=[V\{f(x)|f\in\C_0(\mathbb{R})\}V^*\C^*(\Gamma)\C_0(\mathcal{H})]\\
&=[V\{f(x)|f\in\C_0(\mathbb{R})\}\C^*(\Gamma)\C_0(\mathcal{H})]\\
&=[V\{f(x)|f\in\C_0(\mathbb{R})\}\C_0(\mathcal{H})\rtimes\Gamma]\\
&=[V\C_0(\mathcal{H})\rtimes\Gamma]=\C_0(\mathcal{H})\rtimes\Gamma.
\end{align*}
In the fifth equality we used the fact that $x\in(\C_0(\mathcal{H})\rtimes\Gamma)^\eta$ while in the third and the sixth equality we used the unitarity of $V$. We see that $\pi\in\Mor(\C_0(\mathbb{R}),\C_0(\mathcal{H})^\Psi)$, hence  $\hx=\pi(\id)$ is affiliated with $\C_0(\mathcal{H})^\Psi$.

In a similar way one can also prove   that $\hy$ and $\hw$ are affiliated with $\C_0(\mathcal{H})^\Psi$. Our next objective is to show that $\hx,\hy$ and $\hw$ generate $\CHP$ in the sense of Woronowicz. This follows from the fact that the subset of $\M(\CHP)$ given by
\[\{f_1(\hx)f_2(\hy)f_3(\hw)|\,f_1,f_2\in\C_0(\mathbb{R}),\,f_3\in\C_0(\mathbb{C})\} \] is in fact a linearly dense subset of $\CHP$. To prove this density we use the same arguments that were used in the proof of Theorem 5.5 of \cite{Kasp}. 
\end{proof}
Let us move on to the analysis of the commutation relations for $\hx$, $\hy$ and $\hw$. It is easy to see that $\hx$ and $\hy$  strongly commute. We shall show that the relations between $\hx$ and $\hw$ and between $\hy$ and $\hw$ are of the $(p,q)$-type in the sense of Definition \ref{mu2}. 
\begin{twr}\label{comrelxyw}
Let $\hx,\hy,\hw$ be the generators of $\CHP$ introduced above and $t=e^{-8s}$ where $s\in\mathbb{R}$ is the deformation parameter that specifies the $2$-cocycle \eqref{twococycle}. Then $(\hx,\hw)$ and $(\hy,\hw)$ are respectively a $(t^{-1},t)$ and $(t,t^{-1})$-commuting pair of normal elements affiliated with $\CHP$. 
\end{twr}
\begin{proof}
Let us introduce an affiliated element $T\in(\CHRC)^\eta$ such that $\lambda_z=\exp(i\Im(zT))$. Using the fact that $\lambda$ implements the action of $\mathbb{C}$ on $\C(\mathcal{H})$ we get:
\begin{align}\label{eqw}e^{\bar{z}-z}w&=\exp(i\Im(zT))w \exp(-i\Im(zT)),\\
e^{\bar{z}-z}w&=\exp(i\Im(zT^*))w \exp(-i\Im(zT^*)).
\end{align}
In particular, the affiliated element $(T-T^*)\in(\CHRC)^\eta$ strongly commutes with $w\in(\CHRC)^\eta$. The unitary $V$ (see \eqref{defcoef}) can be expressed by $T$: $V=\exp(-i\frac{s}{2}\Im(T^2))$. Using this we see that:
\begin{align*}\hw&=e^{2s}\exp\left(i\frac{s}{2}\Im(T^2)\right)w \exp\left(-i\frac{s}{2}\Im(T^2)\right)\\
&=e^{2s}\exp\left(i\frac{s}{2}\Im(T(T-T^*))\right)w \exp\left(-i\frac{s}{2}\Im(T(T-T^*))\right)\\&=e^{2s}\exp(-sT+sT^*)w=e^{2s}w\exp(-2is\Im(T)).
\end{align*}
In the second equality we used the fact that $\Im(TT^*)=0$ and in the third equality we used Eq. \eqref{eqw} with $z$ replaced by $\frac{s}{2}(T-T^*)$ (this is legitimate since $T-T^*$ and $w$ strongly commute).
The product $w\exp(-is\Im(T))$ is well defined due to the fact that $w$ and $T-T^*$ strongly commute (see Theorem \ref{prodpq}).
Using the above considerations 
and the easy to check equality \[\exp(-2is\Im(T))z(\hx)\exp(2is\Im(T))=z_{e^{-4s}}(\hx)\] we see that:
\begin{align*}z(\hw)z(\hx)&=
z_{e^{2s}}(w)\exp(-2is\Im(T))z(\hx)\\&=z_{e^{2s}}(w)z_{e^{-4s}}(\hx)\exp(-2is\Im(T))\\
&=z_{e^{2s}}(w)Vz_{e^{-6s}}(x)V^*\exp(-2is\Im(T))\\&=Vz_{e^{2s}}(w)\exp(-2is\Im(T))z_{e^{-6s}}(x)V^*\exp(-2is\Im(T))\\&=
Vz_{e^{2s}}(w)z_{e^{-10s}}(x)V^*\exp(-4is\Im(T))\\&=Vz_{e^{-10s}}(x)V^*z_{e^{2s}}(w)\exp(-2is\Im(T))=z_{e^{-8s}}(\hx)z(\hw).
\end{align*}
This shows that $\hx$ and $\hw$ satisfy the second identity of Definition \ref{mu2} of a $(t,t^{-1})$-commuting pair. Using the fact that $\hx$ is self-adjoint and taking the adjoint of the above calculation we may see that $(\hx,\hw)$ is in fact  an example of $(t,t^{-1})$-commuting pair of normal elements. 

A similar reasoning shows that the pair $(\hy,\hw)$ is an example of a $(t^{-1},t)$-commuting pair. 
\end{proof}

\end{subsection}

\begin{subsection}{Coaction of $\mathbb{G}^\Psi$ on $\CHP$}\label{cosub}
Let $\mathbb{G}^\Psi$ be the quantum Lorentz group described in Theorem \ref{generate}. 
From the results of Section \ref{rdc} we know that there exists a  continuous right coaction $\DHP$  of $\mathbb{G}^\Psi$ on $\CHP$.
The aim of this section is to describe $\DHP$ in terms of its action on generators $\hx,\hy,\hw\in(\CHP)^\eta$. 

The coaction $\Dh$ of $(\C_0(G),\Delta)$ on $\CH$ when applied to generators $x,y,w\in\CH^\eta$ gives
\begin{equation}\label{coactform}\begin{array}{rl}\Dh(x)&\hspace*{-0,25cm}=x\otimes\alpha^*\alpha+w\otimes\alpha^*\gamma+w^*\otimes\gamma^*\alpha+y\otimes \gamma^*\gamma,\\
\Dh(y)&\hspace*{-0,25cm}=x\otimes\beta^*\beta+w\otimes\beta^*\delta+w^*\otimes\delta^*\beta+y\otimes\delta^*\delta,\\
\Dh(w)&\hspace*{-0,25cm}=x\otimes\alpha^*\beta+w\otimes\alpha^*\delta+w^*\otimes\gamma^*\beta+y\otimes\gamma^*\delta.
\end{array}
\end{equation} In what follows we shall show that in the case of $\DHP$ the only change is that one has to add a hat over each affiliated element above.
\begin{twr} Let $\mathbb{G}^\Psi$ be the quantum group described in Theorem \ref{generate}, $\CHP$ be the $\C^*$-algebra described in Theorem \ref{chgen} and $\DHP$ be the coaction of $\mathbb{G}^\Psi$ on $\CHP$ described in the beginning of Section \ref{cosub}.
The action of $\DHP$ on the generators $\hx,\hw,\hy\in(\CHP)^\eta$ is given by 
\begin{align*}
\DHP(\hx)&=\hx\otimes\ha^*\ha+\hw\otimes\ha^*\hg+\hw^*\otimes\hg^*\ha+\hy\otimes \hg^*\hg,\\
\DHP(\hw)&=\hx\otimes\ha^*\hb+\hw\otimes\ha^*\hd+\hw^*\otimes\hg^*\hb+\hy\otimes\hg^*\hd,\\
\DHP(\hy)&=\hx\otimes\hb^*\hb+\hw\otimes\hb^*\hd+\hw^*\otimes\hd^*\hb+\hy\otimes\hd^*\hd,
\end{align*}
where on the right hand side of each of these equalities we have the sums of strongly commuting elements affiliated with $\CHP\otimes\CPG$.
\end{twr}
\begin{proof}
In the course of the proof we shall use the  affiliated element $T\in(\CHRC)^\eta$ such that $\lambda_z=\exp(i\Im(zT))$ (see the proof of Theorem \ref{comrelxyw}). We shall also use $T_l,T_r\in(\CGRC)^\eta$ such that $\lambda_{z_1,z_2}=\exp(i\Im(z_1T_l+z_2T_r))$. 
The generators $\hx,\hy,\hw$ of $\CHP$ may be expressed in terms of $T$ and the coordinates $x,y,w\in\CH^\eta$: \begin{align}\label{defhx}\hx&=e^{-2s}\exp\left(-i\frac{s}{2}\Im(T^2)\right)x\exp\left(i\frac{s}{2}\Im(T^2)\right),\\
\label{defhy}\hy&=e^{-2s}\exp\left(-i\frac{s}{2}\Im(T^2)\right)y\exp\left(i\frac{s}{2}\Im(T^2)\right),\\
\label{defhw}\hw&=e^{2s}\exp\left(i\frac{s}{2}\Im(T^2)\right)w\exp\left(-i\frac{s}{2}\Im(T^2).\right)\end{align}
Similarly, the generators $\ha,\hb,\hg,\hd$ of $\CPG$ may be expressed in terms of $T_l$, $T_r$ and $\alpha,\beta,\gamma,\delta$:
\begin{align}\label{defha}
\ha&=\exp\left(-is\Im(T_rT_l^*)\right)\alpha\exp\left(is\Im(T_rT_l^*)\right),\\
\label{defhb}\hb&=\exp\left(is\Im(T_rT_l^*)\right)\beta\exp\left(-is\Im(T_rT_l^*)\right),\\
\label{defhg}\hg&=\exp\left(is\Im(T_rT_l^*)\right)\gamma\exp\left(-is\Im(T_rT_l^*)\right),\\
\label{defhd}\hd&=\exp\left(-is\Im(T_rT_l^*)\right)\delta\exp\left(is\Im(T_rT_l^*)\right).
\end{align}

Let us move on to the proof of the equality \begin{equation}\label{chdx}\Dh^\Psi(\hx)=\hx\otimes\ha^*\ha+\hw\otimes\ha^*\hg+\hw^*\otimes\hg^*\ha+\hy\otimes \hg^*\hg.\end{equation} From the definition of $\Dh^\Psi$ (see Section \ref{rdc}) we get:
\begin{equation}\label{dx}\Dh^\Psi(\hx)=e^{-2s}Y(x\otimes\alpha^*\alpha+w\otimes\alpha^*\gamma+w^*\otimes\gamma^*\alpha+y\otimes \gamma^*\gamma)Y^*\end{equation}
where $Y\in\M(\CHRC\otimes\CGRC)$ is the unitary element given by \[Y=\exp\left(-i\frac{s}{2}\Im(1\otimes T_r^2+2T\otimes T_l^*)\right).\]
In what follows we shall analyze the four terms appearing in \eqref{dx}, proving that they are equal to the corresponding four terms appearing on the right side of \eqref{chdx}. Let us first show that 
\begin{equation}\label{dx1}e^{-2s}Y(x\otimes\alpha^*\alpha)Y^*=\hx\otimes\ha^*\ha.\end{equation}
Substituting the formula \eqref{defhx} for $\hx$ and the formula \eqref{defha} for  $\ha$ in Eq. \eqref{dx1} we get an equivalent form of \eqref{dx1}:
\[X(x\otimes\alpha^*\alpha)X^*=(x\otimes\alpha^*\alpha),\]
where the unitary element  $X\in\M(\CHRC\otimes\CGRC)$ is given by \[\exp\left(i\frac{s}{2}\Im(T^2\otimes 1-2T\otimes T^*_l-1\otimes (T_r^2-2T_rT_l^*))\right).\]
 It is easy to check that 
\begin{equation}\label{dx1.5} T^2\otimes 1-2T\otimes T^*_l-1\otimes (T_r^2-2T_rT_l^*)=(T\otimes 1-1\otimes T_l^*)^2-1\otimes(T_r-T_l^*)^2.\end{equation}
The element
$T\otimes 1-1\otimes T_l^*$ strongly commutes with $x\otimes\alpha^*\alpha$, which follows from the equality: 
\begin{equation}\label{dx2}\Big(\exp(i\Im(zT))\otimes\exp(-i\Im(zT_l^*))\Big)(x\otimes\alpha^*\alpha)\Big(\exp(-i\Im(zT))\otimes\exp(i\Im(zT_l^*))\Big)=(x\otimes\alpha^*\alpha).\end{equation}
Similarly, the element $1\otimes(T_r-T_l^*)$ strongly commutes with $x\otimes\alpha^*\alpha$ and  using Eq. \eqref{dx1.5} we see that \eqref{dx1} is satisfied.

 Our next objective is to prove that:
\begin{equation}\label{dx3}e^{-2s}Y(w\otimes\alpha^*\gamma)Y^*=\hw\otimes\ha^*\hg.\end{equation}
To this end, let us note that
\[2T\otimes T_l^*=(T\otimes 1+1\otimes T_l^*)^2-T^2\otimes 1-1\otimes T_l^{*2}.\]
It can be checked that $T\otimes 1+1\otimes T_l^*$ strongly commutes with $w\otimes\alpha^*\gamma$, hence Eq. \eqref{dx3} is equivalent to
\[e^{-2s}Y'(w\otimes\alpha^*\gamma)Y'^*=\hw\otimes\ha^*\hg,\] where $Y'=\exp\left(-i\frac{s}{2}\Im(1\otimes T_r^2-1\otimes T_l^{*2}-T^2\otimes 1)\right)$.
Using formula \eqref{defhw} for $\hw$ we see that \[e^{-2s}Y'(w\otimes\alpha^*\gamma)Y'^*=\hw\otimes\exp(-4s)\exp\left(-i\frac{s}{2}\Im(T_r^2- T_l^{*2})\right)\alpha^*\gamma\exp\left(i\frac{s}{2}\Im(T_r^2-T_l^{*2})\right).\] Therefore to show \eqref{dx3} we have to check that \begin{equation}\label{dx4}\exp(-4s)\exp\left(-i\frac{s}{2}\Im(T_r^2-T_l^{*2})\right)\alpha^*\gamma\exp\left(i\frac{s}{2}\Im(T_r^2- T_l^{*2})\right)=\ha^*\hg.\end{equation} 
The identity \[\exp(i\Im(zT_l^*))\alpha^*\gamma\exp(-i\Im(zT_l^*))=\exp(z-\bar{z})T_l\] with $z$ replaced by $\frac{s}{2}(T_l^*-T_l)$ (this is legitimated by the strong commutativity of $\alpha^*\gamma$ and $T_l^*-T_l$) shows that
\begin{equation}\label{dx5}\exp\left(i\frac{s}{2}\Im(T_l^{*2})\right)\alpha^*\gamma\exp\left(-i\frac{s}{2}\Im(T_l^{*2})\right)=\exp(sT_l^*-sT_l)\alpha^*\gamma.\end{equation} Similarly, using the identity \[\exp(i\Im(zT_r))\alpha^*\gamma\exp(-i\Im(zT_r))=\exp(z+\bar{z})T_l\] with $z$ replaced by $-\frac{s}{2}(T_r^*+T_r)$ (which in turn  is legitimated by the strong commutativity of $\alpha^*\gamma$ and $T_r^*+T_r$) we get 
\begin{equation}\label{dx6}\exp\left(-i\frac{s}{2}\Im(T_r^2)\right)\alpha^*\gamma\exp\left(i\frac{s}{2}\Im(T_r^2)\right)=\exp(-sT_r^*-sT_r)\alpha^*\gamma.\end{equation}
Equations \eqref{dx5} and \eqref{dx6} together show that the left hand side of \eqref{dx4} is given by 
\[\exp(-4s)\exp(s(T_l^*-T_l-T_r^*-T_r))\alpha^*\gamma.\] This is equal to the right hand side of \eqref{dx4}, because \begin{equation}\label{dx7}\ha^*\hg=\exp(-s(T_l+T_r))\alpha^*\exp(s(T_l^*-T_r^*))\gamma=\exp(-4s)\exp(s(T_l^*-T_l-T_r^*-T_r))\alpha^*\gamma.\end{equation}
To prove the above equality we first use the identity 
\begin{equation}\label{dx7.1}\ha=\exp(-is\Im(T_rT_l^*))\alpha\exp(is\Im(T_rT_l^*))=\exp(-sT_l^*-sT_r^*)\alpha\end{equation} which in turn follows from the equality 
\[\exp(i\Im(zT_r)\alpha\exp(-i\Im(zT_r)=\exp(z)\alpha\] with $z$ replaced $-sT_l^*-sT_r^*$ (as the reader may expect, one has to invoke at this moment the strong commutativity of $\alpha$ and $T_l+T_r$ to legitimate this argument). Similarly, we can show that
\[\hg=\exp(sT_l^*-sT_r^*)\gamma.\] Finally, in the second equality of \eqref{dx7} we have used the fact \[\alpha^*\exp(s(T_l^*-T_r^*))=\exp(-4s)\exp(s(T_l^*-T_r^*))\alpha^*,\] which can be proved using the framework of $(p,q)$-commuting pairs of elements affiliated with $\C_0(G)\rtimes\mathbb{C}^2$. This ends the proof of \eqref{dx3}, which by taking the adjoint also shows that 
\begin{equation}\label{dx8}e^{-2s}Y(w^*\otimes\gamma^*\alpha)Y^*=\hw^*\otimes\hg^*\ha.\end{equation} Let us now check that 
\begin{equation}\label{dx9} e^{-2s}Y(y\otimes\gamma^*\gamma)Y^*=\hy\otimes\hg^*\hg.\end{equation} The reasoning is similar to the one which proved \eqref{dx1}. We begin by substituting formula \eqref{defhy} for  $\hy$ and \eqref{defhg} for $\hg$ in Eq. \eqref{dx9} to obtain the equivalent equality:
\begin{equation}\label{dx10}X(y\otimes\gamma^*\gamma)X^*=(y\otimes\gamma^*\gamma).\end{equation}
In this case $X\in\M(\CHRC\otimes\CGRC)$ is given by \[\exp\left(i\frac{s}{2}\Im(T^2\otimes 1-2T\otimes T^*_l-1\otimes (T_r^2+2T_rT_l^*))\right).\]
Let us note that
\[ T^2\otimes 1-2T\otimes T^*_l-1\otimes (T_r^2+2T_rT_l^*)=(T\otimes 1-1\otimes T_l^*)^2-1\otimes(T_r+T_l^*)^2.\]
Equation \eqref{dx10} follows from the fact that both operators $T\otimes 1-1\otimes T_l^*$ and $1\otimes(T_r+T_l^*)$ strongly commute with $y\otimes \gamma^*\gamma$. 
Eq. \eqref{dx1} together with \eqref{dx3}, \eqref{dx8} and \eqref{dx9} proves \eqref{chdx}.
It follows from construction that the sum on the right is the sum of strongly commuting, normal elements affiliated with $\CHP\otimes\CPG$. 

Let us move on to the proof of the second equality of our theorem: 
\begin{equation}\label{dwch}
\DHP(\hw)=\hx\otimes\ha^*\hb+\hw\otimes\ha^*\hd+\hw^*\otimes\hg^*\hb+\hy\otimes\hg^*\hd,
\end{equation}
By the definition of $\DHP$ (see Section \ref{rdc}) it may be seen that
\begin{equation}\label{dw}\DHP(\hw)=e^{2s}Y(x\otimes\alpha^*\beta+w\otimes\alpha^*\delta+w^*\otimes\gamma^*\beta+y\otimes\gamma^*\delta)Y^*,
\end{equation}
where in this case $Y\in\M(\CHRC\otimes\CGRC)$ is a unitary element of the following form 
\begin{equation}\label{dw0.5}Y=\exp\left(i\frac{s}{2}\Im(1\otimes T_r^2-2T\otimes T_l^*)\right).\end{equation}
As before, we shall analyze the four terms appearing on the right hand side of Eq. \eqref{dw}. In order to show that
\begin{equation}\label{dw1} e^{2s}Y(x\otimes\alpha^*\beta)Y^*=\hx\otimes\ha^*\hb
\end{equation}
we use the equality:
\begin{equation}\label{dw3.5}-2T\otimes T_l^*=(T\otimes 1-1\otimes T_l^*)^2-T^2\otimes 1-1\otimes T_l^{*2}.\end{equation} The fact that 
$T\otimes 1-1\otimes T_l^*$ strongly commutes with $x\otimes\alpha^*\beta$ leads to the following equality (see also \eqref{dx4}):
\begin{equation}\label{dw2}
\exp(4s)\exp\left(i\frac{s}{2}\Im(T_r^2-T_l^{*2})\right)\alpha^*\beta\exp\left(-i\frac{s}{2}\Im(T_r^2- T_l^{*2})\right)=\ha^*\hb, 
\end{equation} which we check below.
Using 
\[\exp(i\Im(zT_r))\alpha^*\beta\exp(-i\Im(zT_r))=e^{\bar{z}-z}\alpha^*\beta\] with $z$ replaced by $\frac{s}{2}(T_r-T_r^*)$ we get 
\begin{equation}\label{dw3}\exp\left(i\frac{s}{2}\Im(T_r^2)\right)\alpha^*\beta\exp\left(-i\frac{s}{2}\Im(T_r^2)\right)=\exp(sT_r^*-sT_r)\alpha^*\beta.\end{equation}
Similarly, we may show that 
\begin{equation}\label{dw4}
\exp\left(-i\frac{s}{2}\Im(T_l^{*2})\right)\alpha^*\beta\exp\left(i\frac{s}{2}\Im(T_l^{*2})\right)=\exp(-sT_l-sT_l^*)\alpha^*\beta.
\end{equation}
Eqs. \eqref{dw3} and \eqref{dw4} show that the left hand side of \eqref{dw2} is equal to \begin{equation}\label{dw5}e^{4s}\exp(sT_r^*-sT_r-sT_l^*-sT_l)\alpha^*\beta\end{equation} Eq. \eqref{dx7.1} gives 
\[\ha^*=\exp(-sT_l-sT_r)\alpha^*.\] Similarly, using \eqref{defhb} we can check that \[\hb=\exp(sT_r^*-sT_l^*)\beta.\]
The above two equalities give:
\begin{equation}\label{dw6}\ha^*\hb=\exp(-sT_l-sT_r)\alpha^*\exp(sT_r^*-sT_l^*)\beta=e^{4s}\exp(sT_r^*-sT_r-sT_l^*-sT_l)\alpha^*\beta,\end{equation}
where in the final step we used \[\alpha^*\exp(sT_r^*-sT_l^*)=e^{4s}\exp(sT_r^*-sT_l^*)\alpha^*.\] The last formula can be proved in the framework of $(p,q)$-commuting pairs of elements affiliated with $\CGRC$ (see Theorem \ref{twrs}). Using \eqref{dw5} and \eqref{dw6}  we get \eqref{dw1}. 

Our next objective is to prove that
\begin{equation}\label{dw7} e^{2s}Y(w\otimes\alpha^*\delta)Y^*=\hw\otimes\ha^*\hd.
\end{equation} Inserting formula \eqref{defhw}, \eqref{defha} and \eqref{defhd} for $\hw$, $\ha$ and $\hd$ respectively  we see that the above equality is equivalent with the following one 
\begin{equation}\label{dw8}
Y'(w\otimes\alpha^*\delta)Y'^*=(w\otimes\alpha^*\delta),
\end{equation}
where $Y'\in\M(\CHRC\otimes\CGRC)$ is a unitary element given by \[Y'=\exp\left(i\frac{s}{2}\Im(1\otimes T_r^2-2T\otimes T_l^*-T^2\otimes1+2\otimes T_rT_l^*)\right).\] It is easy to see that
\[1\otimes T_r^2-2T\otimes T_l^*-T^2\otimes1+2\otimes T_rT_l^*=(1\otimes T_r+1\otimes T_l^*)^2-(T\otimes 1+1\otimes T_l^*)^2.\] 
Eq.
\eqref{dw8} follows from the observation that both elements 
$1\otimes(T_r+T_l^*)$ and $(T\otimes 1+1\otimes T_l^*)$ strongly commute with $(w\otimes\alpha^*\delta)$. 

The proof of the equality 
\begin{equation}\label{dw9} e^{2s}Y(w^*\otimes\gamma^*\beta)Y^*=\hw^*\otimes\hg^*\hb
\end{equation} is similar to the proof of \eqref{dw7}. It is based on the observation that the element 
\[1\otimes T_r^2-1\otimes 2T_rT_l^*-T^2\otimes 1-2T\otimes T_l^*=1\otimes (T_r-T_l^*)^2-(T\otimes 1+1\otimes T_l^*)^2\] strongly commutes with $w^*\otimes\gamma^*\beta$.

Finally, let us check that 
\[e^{2s}Y(y\otimes\gamma^*\delta)Y^*=\hy\otimes\hg^*\hd,\] where $Y$ is given by \eqref{dw0.5}. 
The following reasoning is similar to the one which proved \eqref{dw1}. Using 
\eqref{dw3.5} together with the strong commutativity  of  $T\otimes 1-1\otimes T_l^*$ and  $y\otimes\gamma^*\delta$ we get an equivalent formula, which is 
\begin{equation}\label{dw10}
\exp(4s)\exp\left(i\frac{s}{2}\Im(T_r^2-T_l^{*2})\right)\gamma^*\delta\exp\left(-i\frac{s}{2}\Im(T_r^2- T_l^{*2})\right)=\hg^*\hd.
\end{equation}
It can be shown (compare with \eqref{dw3} and \eqref{dw4}) that 
\begin{equation}\label{dw11}
\exp\left(i\frac{s}{2}\Im(T_r^2)\right)\gamma^*\delta\exp\left(-i\frac{s}{2}\Im(T_r^2)\right)=\exp(sT_r^*-sT_r)\gamma^*\delta.
\end{equation}
\begin{equation}\label{dw12}
\exp\left(-i\frac{s}{2}\Im(T_l^{*2})\right)\gamma^*\delta\exp\left(i\frac{s}{2}\Im(T_l^{*2})\right)=\exp(sT_l+sT_l^*)\gamma^*\delta.
\end{equation}
These two identities show that the left hand side of \eqref{dw10} is given by 
\[e^{4s}\exp(sT_l+sT_l^*+sT_r^*-sT_r)\gamma^*\delta\] and moreover, it is equal to the right hand side by the following computation (compare with \eqref{dw6})
\[\hg^*\hd=\exp(sT_l-sT_r)\gamma^*\exp(sT_r^*+sT_l^*)\delta=e^{4s}\exp(sT_l+sT_l^*+sT_r^*-sT_r)\gamma^*\delta.\]
In the second equality we used the fact that \[\gamma^*\exp(sT_r^*+sT_l^*)=e^{4s}\exp(sT_r^*+sT_l^*)\gamma^*,\] which can be proved in the framework of $(p,q)$-commuting pairs of elements affiliated with  $\CGRC$ (see Theorem \ref{twrs}).
The equalities \eqref{dw1}, \eqref{dw7}, \eqref{dw9} and \eqref{dw10} together with \eqref{dw} imply \eqref{dwch}.

The proof of the equality 
\[\DHP(\hy)=\hx\otimes\hb^*\hb+\hw\otimes\hb^*\hd+\hw^*\otimes\hd^*\hb+\hy\otimes\hd^*\hd,\]  
is very similar to the proof of \eqref{chdx} and is left to the reader.
\end{proof}
 \end{subsection}
\end{section}

\end{document}